\documentclass{amsart}

\usepackage{amsmath, amssymb}
\usepackage{amscd}
\usepackage[arrow,matrix]{xy}
\usepackage{hyperref}           
\newtheorem{definition}{Definition}[section]
\newtheorem{theorem}[definition]{Theorem}
\newtheorem{lemma}[definition]{Lemma}
\newtheorem{corollary}[definition]{Corollary}
\newtheorem{proposition}[definition]{Proposition}
\theoremstyle{definition}
\newtheorem{remark}[definition]{Remark}

\newtheorem{example}[definition]{Example}

\newtheorem{notation}[definition]{Notation}

\newcommand\CC{\mathbb{C}}
\newcommand\NN{\mathbb{N}}

\newcommand\TT{\mathbb{T}}

\newcommand\ds{\displaystyle}

\newcommand\style{\mathcal }          

\newcommand{\B}{\style{B}}
\newcommand{\M}{M}
\renewcommand{\H}{\style{H}}
\newcommand{\K}{\style K}

\newcommand{\X}{\style{X}}

\newcommand{\C}{\mathbb{C}}

\newcommand\shilov{\partial_{\rm S}}   

\newcommand\choquet{\partial_{\rm C}}

\newcommand\osf{{\style F}}

\newcommand\osp{{\style P}}
\newcommand\oss{{\style S}}
\newcommand\ost{{\style T}}

\newcommand\csta{{\style A}}

\newcommand\cstar{{\rm C}^*}                              
\newcommand\cstare{{\rm C}_{\rm e}^*}              

\newcommand\conv{{\rm Conv} }                          

\begin{document}

\title[The Noncommutative Choquet Boundary of Periodic Weighted Shifts]{The Noncommutative Choquet Boundary of Periodic Weighted Shifts}

\author[Mart\'\i n Argerami]{Mart\'\i n~Argerami}
\address{Department of Mathematics and Statistics,
University of Regina,
Regina, SK S4S 0A2,
Canada}
\email{argerami@math.uregina.ca}

\author[Douglas Farenick]{Douglas Farenick}
\address{Department of Mathematics and Statistics, University of Regina,
Regina, SK S4S 0A2, Canada}
\email{douglas.farenick@uregina.ca}

\thanks{This work is supported in part by the NSERC Discovery Grant program}

\begin{abstract}
The noncommutative Choquet boundary and the C$^*$-envelope of operator systems of the form $\mbox{\rm Span}\,\{1,T,T^*\}$,
where $T$ is a Hilbert space operator with normal-like features, are studied.
Such operators include normal operators, $k$-normal operators, subnormal operators, and Toeplitz operators. Our main result is the determination of the
noncommutative Choquet boundary for an operator system generated by an irreducible periodic weighted unilateral shift operator.
\end{abstract}

\maketitle

\section{Introduction}

If $Y$ is a compact Hausdorff space and $C(Y)$ is the Banach space of all continuous complex-valued functions on $Y$, then the \emph{Choquet boundary} of a
linear subspace $\mathcal F\subset C(Y)$ that contains the constants and separates the points of $Y$ is the subset $\choquet\mathcal F\subset Y$ of all $y\in Y$
for which the point-mass measure $\delta_y$ on the Borel sets of $Y$ is the only Borel probability measure $\mu$ on $Y$ for which $f(y)=\int_X f\,d\mu$ for every
$f\in\mathcal F$.
Motivated by the use of the Choquet boundary in the analysis of spaces of continuous complex-valued functions (as in \cite{Phelps-book}, for example),
W.~Arveson initiated the study of analogous objects in the setting of matricially ordered
vector spaces $\X$ of bounded linear operators acting
on complex Hilbert spaces $\H$ \cite{arveson1969,arveson1972}.

A notion that is central in Arvesons's work and its subsequent application is that of a boundary representation.
A  \emph{boundary representation} for a unital operator space $\X\subset\B(\H)$---that is, a subspace $\X\subset\B(\H)$ with $1_{\B(\H)}\in\X$---is a unital C$^*$-algebra
representation $\rho:\cstar(\X)\rightarrow\B(\H_\rho)$ such that
\begin{enumerate}
\item $\rho$ is irreducible and
\item for any unital completely positive (ucp) linear map $\psi:\cstar(\X)\rightarrow\B(\H_\rho)$ with $\psi|_\X^{\phantom{\X}}=
\rho|_\X^{\phantom{\X}}$ we have $\psi=\rho$ (i.e. $\rho|_\X^{\phantom{\X}}$ has a unique completely contractive
extension to $\cstar(\X)$, namely, $\rho$).
\end{enumerate}

If $\choquet\X$ denotes the subset of the spectrum of $\cstar(\X)$ consisting of the
unitary-equivalence classes $\dot{\rho}$ of boundary representations
$\rho$ for $\X$, then
the ideal $\mathfrak S_\X\subset\cstar(\X)$ defined by
\[
\mathfrak S_\X\,=\,\bigcap_{\dot{\rho}\in\choquet\X}\,\ker\rho\,,
\]
is called the \emph{\v Silov ideal} for $\X$. A unital operator space $\X\subset\B(\H)$
is said to \emph{have a noncommutative Choquet boundary} if
the canonical quotient homomorphism $\cstar(\X)\rightarrow\cstar(\X)/\mathfrak S_\X$ is a complete isometry on $\X$,
in which case the subset $\choquet\X$ of the spectrum of $\cstar(\X)$ is called the \emph{noncommutative Choquet boundary} of $\X$. If $\X$ has a noncommutative Choquet
boundary, then the C$^*$-algebra
$\cstar(\X)/\mathfrak S_\X$ is called the \emph{{\rm C}$^*$-envelope} of $\X$, which we denote by $\cstare(\X)$.
An important theorem of Arveson \cite{arveson2008} asserts that every
separable unital operator space $\X\subset\B(\H)$ has a noncommutative Choquet boundary.

The C$^*$-envelope $\cstare(\X)$ of a unital operator space $\X$
can be viewed as the smallest C$^*$-algebra that is generated by (a copy of) $\X$; concretely, $(\cstare(\X),\iota)$
is a $\cstar$-envelope for $\X$ if $\cstare(\X)$ is a $\cstar$-algebra and $\iota:\X\to\cstare(\X)$ is a complete
isometry such that
whenever $\phi:\X\to A$ is
a complete isometry into a $\cstar$-algebra $A$,  there exists an epimorphism of $\cstar$-algebras
$\pi:\cstar(\phi(\X))\to\cstare(\X)$ such that $\pi\circ\phi=\iota$. It follows easily from this definition
that the $\cstar$-envelope of $\X$ is unique up to isomorphism of $\cstar$-algebras.

It is not easy, in general, to determine the $\cstar$-envelope of a given unital operator space.
There is, however, a substantial literature
for the case in which $\X$ is an operator algebra
(for example, \cite{arveson1998,Blecher--LeMerdy-book,davidson--katsoulis2011,katsoulis--kribs2006b}).
But our focus in this paper is in the realm of single operator theory, as we consider the smallest possible
unital operator spaces (namely, those of the form $\X_T:=\mbox{Span}\,\{1,T\}$ for some operator $T\in\B(\H)$).

Because completely contractive linear maps $\phi_0^{\phantom{0}}:\X\rightarrow\B(\K)$ of a unital operator space $\X\subset\B(\H)$
extend to completely positive linear maps $\phi$ of the operator system $\X+\X^*$ via
$\phi_0^{\phantom{0}}(X+Y^*)=\phi_0^{\phantom{0}}(X)+\phi_0^{\phantom{0}}(Y)^*$ \cite[Proposition 2.12]{Paulsen-book}, $\X$ and $\X+\X^*$ have the same boundary representations.
Therefore, we shall mainly consider operator systems. In particular, if $T\in\B(\H)$, then
\[
\oss_T\,=\,\mbox{\rm Span}\,\{1,T,T^*\}
\]
is the \emph{operator system generated by $T$}. Such an operator system generates a unital C$^*$-algebra, namely $\cstar(\oss_T)$, which we
sometimes denote by $\cstar(T)$.

A linear map $\phi:\oss\rightarrow\ost$ of operator systems $\oss$ and $\ost$ is a \emph{complete order isomorphism} if $\phi$ is a linear isomorphism and both
$\phi$ and $\phi^{-1}$ are completely positive. If $\phi$ is a unital linear isomorphism, then $\phi$ is a complete order isomorphism if and only if $\phi$ is a
complete isometry  \cite[Proposition 13.3]{Paulsen-book}.
A straightforward application of the definition of the $\cstar$-envelope shows that
every unital complete order isomorphism $\phi:\oss\rightarrow\ost$ extends
to a unital C$^*$-algebra homomorphism $\pi:\cstare(\oss)\rightarrow\cstare(\ost)$.

Our primary objective is to determine $\choquet\oss_T$ and $\cstare(\oss_T)$ for irreducible periodic weighted unilateral shift operators $T$.
But to do so, it is useful to first consider the case of normal operators,
followed by the case of operators that display normal-like properties. While the results on normal operators are surely known to experts, we have not seen an explicit discussion of this case in the
literature, and so we present a self-contained treatment here, giving particular emphasis to the role of the numerical range and spectrum in obtaining such results.

\section{Preliminaries}

\begin{notation}
For an operator system $\oss\subset\B(\H)$,
the canonical quotient homomorphism $\cstar(\oss)\rightarrow\cstar(\oss)/\mathfrak S_\oss$
is denoted by $q_{\rm e}$, and  $\iota_{\rm e}=q_{\rm e}{}|_{\oss}$ denotes the ucp map $\oss\rightarrow
\cstar(\oss)/\mathfrak S_\oss$.
\end{notation}

Note that, by definition,  $\oss$ has a noncommutative Choquet boundary if and only if $\iota_{\rm e}$ is a
completely isometric embedding.

\begin{lemma}\label{matrices} $\cstar(\oss_T\otimes\M_m(\CC))=\cstar(\oss_T)\otimes\M_m(\CC)$
for every $T\in\B(\H)$ and $m\in\mathbb N$.
\end{lemma}

\begin{proof} We can identify $\oss_T\otimes\M_m(\CC)$ canonically with $M_m(\oss_T)$,
and $\cstar(\oss_T)\otimes\M_m(\CC)$
with $M_m(\cstar(\oss_T))$.
And so our assertion reduces to the also canonical identification of $M_m(\cstar(\oss_T))$
with $\cstar(M_m(\oss_T))$.
\end{proof}

The ideal $\mathfrak S_\oss$ is not an invariant of the operator system $\oss$, as it depends
on the concrete representation of $\oss$; still it
plays a crucial role in Arveson's theory. An easy but key fact that motivates
this significance is as follows:

\begin{lemma}\label{lemma: the quotient as a product}
For any $m\in\NN$, there is a canonical isometric embedding {\rm (}as C$^*$-algebras{\rm )}
\[
\xymatrix{
M_m\left(\cstar(\oss)/\mathfrak S_\oss\right) \ar[rr]^{f_m}&  &
\ds\prod_{\dot{\rho}\in\choquet\oss}M_m(\rho(\cstar(\oss)))
}\] \[
\xymatrix{\ \ \ \ \
 X+M_m(\mathfrak S_\oss)\ \ \ \ar@{|->}[rr]& &\ \ \ \ \ \left(\rho^{(m)}(X)\right)_\rho
 \ \ \ \ \ \ \ \ \
}
\]
\end{lemma}
\begin{proof}
We define
\[
f_m(X+M_m(\mathfrak S_\oss))=\prod_{\dot{\rho}\in\choquet\oss}\rho^{(m)}(X),\ \ \ X\in M_m(\cstar(\oss)).
\]
Note that $f_m$ is clearly linear, multiplicative, and $*$-preserving. Also,
\[
\begin{array}{rcl}
\rho^{(m)}(X)=0\ \forall\rho\ &\iff &\rho(X_{hk})=0, \forall \rho, \forall h,k\  \\ &\iff&
X_{hk}\in\mathfrak S_\oss,\ \forall h,k\ \\ &\iff& \ X\in M_m(\mathfrak S_\oss),
\end{array}
\]
which shows that $f_m$ is both well defined and one-to-one.
\end{proof}

We shall also compare our analysis of $\X_T:=\mbox{Span}\,\{1,T\}$ with that of the norm-closed algebra generated by $T$.

\begin{definition} For any operator $T\in\B(\H)$, the \emph{operator algebra generated by $T$}
is the subalgebra $\osp_T\subset\B(\H)$ given by the norm closure of all operators of the form $p(T)$,
for polynomials $p\in\CC\,[t]$.
\end{definition}

As mentioned in the Introduction, most of the focus on $\cstar$-envelopes in the literature is on operator
algebras, while here we focus on operator systems.
Although $\X_T$ and $\osp_T$ generate the same C$^*$-subalgebra of $\B(\H)$,
they do not necessarily have the same C$^*$-envelopes. Explicit examples of this occur in Example
\ref{eg1} when $|\lambda|\leq1/2$, and for operators $T=T^*$ with $|\sigma(T)|\geq3$ (see Proposition
\ref{proposition: selfadjoint case}).

\section{Numerical Range and Spectrum}

We will see below how the numerical range $W(T)$ and spectrum $\sigma(T)$
of $T$ capture information about the boundary representations of $\oss_T$. By numerical range we mean
the compact convex set
\[
W(T)\,=\,\{\phi(T)\,:\,\phi\;\mbox{ is a state on }\oss_T\}\,,
\]
It is well known that the set
\[
W_{\rm s}(T)\,=\,\{\langle T\xi,\xi\rangle\,:\,\xi\in\H,\;\|\xi\|=1\}
\]
is convex and dense in $W(T)$.

\begin{proposition}\label{extremal spectra} Let $T\in\B(\H)$, $\lambda\in\CC$.
\begin{enumerate}
\item\label{ah1} If $\lambda=\rho(T)$ for some $\rho\in\choquet\oss_T$,
then $\lambda\in\sigma(T)\cap\partial W(T)$, and $\lambda$ is an extreme point of $W(T)$.
\item\label{ah2} Assume that $\lambda\in\sigma(T)\cap\partial W(T)$. If $\lambda$ is an extreme point of $W(T)$ and if the commutator $[T^*,T]=T^*T-TT^*$ is positive,
then $\lambda=\rho(T)$ for some $\rho\in\choquet\oss_T$.
\end{enumerate}
\end{proposition}

\begin{proof} To prove \eqref{ah1}, note first that we have $\rho(T-\lambda 1)=0$. As $\rho$ is unital and
multiplicative, this shows that $\lambda\in\sigma(T)$. Also, since $\rho(T)$ is scalar, we
have that $\rho$ is a state on $\oss_T$, and thus $\lambda\in W(T)$. After we prove that $\lambda$
is an extreme point of $W(T)$, we will know that $\lambda\in\partial W(T)$.

Let $\phi=\rho|_{\oss_T}$.
Suppose that $\lambda_1,\lambda_2\in W(T)$ and that $\lambda=\frac{1}{2}\lambda_1+\frac{1}{2}\lambda_2$. As every state on
$\oss_T$ extends to a state on $\cstar(\oss_T)$ (by the Hahn--Banach Theorem and some positivity considerations),
there are states $\phi_1$ and $\phi_2$ on $\cstar(\oss_T)$ such that $\lambda_j=\phi_j(T)$, $j=1,2$. Thus, the state
$\psi=\frac{1}{2}\phi_1+\frac{1}{2}\phi_2$ is an extension of $\phi$ to $\cstar(\oss_T)$. Because $\rho$ is a
boundary representation for $\oss_T$, $\psi=\rho$. That is,
$\rho=\frac{1}{2}\phi_1+\frac{1}{2}\phi_2$. But since $\rho$ is a pure state (because it is
multiplicative),
we deduce that $\phi_1=\phi_2=\rho$; hence, $\lambda_1=\lambda_2=\lambda$, which implies that $\lambda$ is an extreme point of $W(T)$.

For the proof of \eqref{ah2},
the hypothesis $\lambda\in\sigma(T)\cap\partial W(T)$ implies that there is a homomorphism
$\rho:\cstar(\oss_T)\rightarrow\CC$ such that $\lambda=\rho(T)$ \cite[Theorem 3.1.2]{arveson1969}.
Assume that $\lambda$ is an extreme point of $W(T)$ and that $[T^*,T]=T^*T-TT^*$ is positive.
Let $\phi=\rho|_{\oss_T}$ and suppose that $\Phi$ is any state on $\cstar(T)$ that extends $\phi$.
Via the GNS construction, there are a Hilbert space $\H_\pi$, a representation $\pi:\cstar(\oss_T)\rightarrow\B(\H_\pi)$, and
a unit vector $\xi\in\H_\pi$ such that $\Phi(A)=\langle\pi(A)\xi,\xi\rangle$ for every $A\in\cstar(\oss_T)$. In particular, $\lambda=\langle\pi(T)\xi,\xi\rangle$. Now since the
numerical range of $\pi(T)$ is a subset of the numerical range of $T$, $\lambda$ is an extreme point of $W(\pi(T))$. Moreover, as $[\pi(T)^*,\pi(T)]=\pi\left([T^*,T]\right)$
is positive, $W(\pi(T))$ coincides with the convex hull of the spectrum of $\pi(T)$. Hence, the equation $\lambda=\langle\pi(T)\xi,\xi\rangle$ together
with $\lambda\in \sigma\left( \pi(T) \right)\cap \partial W(\pi(T))$ imply that
$\pi(T)\xi=\lambda\xi$ and $\pi(T)^*\xi=\overline\lambda \xi$ \cite[Satz2]{hildebrandt1966}. Thus, $\Phi$ is a homomorphism and agrees with $\rho$ on the generating set $\oss_T$;
hence, $\Phi=\rho$ and so $\rho$ is a boundary representation.
\end{proof}

It is interesting to contrast \eqref{ah1} of Proposition \ref{extremal spectra} with Theorem 3.1.2 of \cite{arveson1969}, which states that if
$\lambda\in\sigma(T)\cap\partial W(T)$, then $\lambda=\rho(T)$ for some boundary representation $\rho$ for $\osp_T$. In this latter assertion,
there is no requirement that $\lambda$ be an extreme point of $W(T)$, and this is one way in which we see that the
operator spaces $\osp_T$ and $\oss_T$ differ fundamentally.

In general a spectral point $\lambda\in\sigma(T)$ that also happens to be an extreme point of $W(T)$ does not give rise to a boundary representation (which explains the extra hypothesis in
assertion \eqref{ah2} of Proposition \ref{extremal spectra}). For example, with $T=\begin{bmatrix}0&1\\0&0\end{bmatrix}\oplus\begin{bmatrix}\frac{1}{2}\end{bmatrix}$,
the vectors $\xi=e_3$ and $\eta=\sqrt{\frac{1}{2}}(e_1+e_2)$ give rise to states $\rho(X)=\langle X\xi,\xi\rangle$ and $\psi(X)=\langle X\eta,\eta\rangle$ on $\cstar(T)$
such that $\rho(T)=\psi(T)=\frac{1}{2}
\in \mbox{ext}\,W(T)\cap\sigma(T)$; however, $\rho$ is a representation of $\cstar(T)$ whereas $\psi$ is not.

\bigskip

In Theorem \ref{theorem: boundary reps and numerical range} below,
numerical range considerations allow us to completely characterise one-dimensional
boundary representations for direct sums
of operators.

\begin{theorem}\label{theorem: boundary reps and numerical range}
Let $T=\ds\bigoplus_{j=1}^mT_j\subset\bigoplus_{j=1}^m\B(\H_j)$, where $m\in\NN$ and
$k_\ell=1$ for a fixed $\ell$. Let $\pi_\ell:\cstar(T)\to\CC$ be the irreducible representation
induced by $\bigoplus T_j\mapsto T_\ell$. Then
$\pi_\ell$ is a boundary representation if and only if $T_\ell\not\in\conv\bigcup_{j\ne\ell}W(T_j)$.
\end{theorem}
\begin{proof}
Let us denote $\lambda=T_\ell$, and $\pi_j:\cstar(T)\to \B(\H_j)$ the representation $\bigoplus T_j
\mapsto T_j$.

Assume first that $\lambda\in\conv\bigcup_{j\ne\ell}W(T_j)$. Therefore, for each $j=1,\ldots,m$ with $j\ne\ell$
there exists a
state $\psi_j$ on $\cstar(T_j)$ such that $\lambda=\sum_{j\ne\ell}\alpha_j\psi_j(T_j)$ for some convex coefficients
$\alpha_1,\ldots,\alpha_{\ell-1},\alpha_{\ell+1}\,\ldots,\alpha_m$. Define a state $\psi$ on $\cstar(T)$
by $\psi=\sum_{j\ne\ell}\alpha_j\psi_j\circ\pi_j$.
Because $\psi(T)=\lambda$, we obtain
$\psi|_{\oss_{T}}=\pi_\ell|_{\oss_{T}}$. Now choose some $j$ with $\alpha_j\ne0$. Then
$\psi(1_j)\geq\alpha_j\psi_j(1_j)=\alpha_j>0$. But, as $j\ne\ell$, $\pi_\ell(1_j)=0$;
thus, $\pi_\ell|_{\oss_{T}}$ admits a ucp extension
from $\oss_{T}$ to $\cstar(\oss_{T})$ other than $\pi_\ell$,
and so $\pi_\ell$ is not a boundary representation for $\oss_{T}$.

Conversely, assume that $\lambda\not\in\conv\bigcup_{j\ne\ell}W(T_j)$.
Choose any state $\phi$ on $\cstar(\oss_{T})$ for which
$\phi|_{\oss_{T}}=\pi_\ell|_{\oss_{T}}$; that  is, $\phi$ is a state such that $\phi(T)=\lambda$.
The numerical range $W(T)$ of $T$ is
the convex hull of the numerical ranges $W(T_1),\ldots,W(T_m)$.
As $W(T_\ell)=\{\lambda\}$ and $\lambda$ is not in the convex hull of the
other numerical ranges we have that $W(T)$ is the convex set generated by the convex set
$\conv\bigcup_{j\ne\ell}W(T_j)$ and the external point $\lambda$; so
$\lambda$ is a point of nondifferentiability on the boundary of $W(T)$.
By the GNS decomposition, there are a Hilbert space $\H_\vartheta$, a representation
$\vartheta:\cstar(\oss_T)\rightarrow\B(\H_\vartheta)$, and
a unit vector $\xi\in\H_\vartheta$ such that $\phi(A)=\langle\vartheta(A)\xi,\xi\rangle$
for every $A\in\cstar(\oss_T)$. In particular, $\lambda=\langle\vartheta(T)\xi,\xi\rangle$.
Because the
numerical range of $\vartheta(T)$ is a subset of the numerical range of $T$, $\lambda$ is also
a point of nondifferentiability on the boundary of $W(\vartheta(T))$; therefore, $\lambda$ is
necessarily an eigenvalue of $\vartheta(T)$ (\cite[Theorem 1]{donoghue1957}).  Moreover, because
this eigenvalue $\lambda$ lies on the
boundary of the numerical range of $\vartheta(T)$,
 the equation $\lambda=\langle\vartheta(T)\xi,\xi\rangle$ implies that
$\vartheta(T)\xi=\lambda\xi$ and $\vartheta(T)^*\xi=\overline\lambda \xi$
\cite[Satz 1,2]{hildebrandt1966}. That is, $\phi$ is a homomorphism and it
agrees with $\pi_\ell$ on $\oss_{T}$;
hence, $\phi=\pi_\ell$ on $\cstar(\oss_{T})$, which proves that $\pi_\ell$ is a boundary representation.
\end{proof}

\begin{remark}
A characterisation of boundary representations of higher order appears in Theorem
\ref{matrix convex pure}. The implications of Theorems
\ref{theorem: boundary reps and numerical range} and \ref{matrix convex pure} to direct sums
of operators and to Jordan operators in particular will be explored in a further article.
\end{remark}

\begin{example}\label{eg1} For each $\lambda\in\CC$, let $T_\lambda\in\M_3(\CC)$ be given by
\[
T_\lambda\,=\,\left[\begin{array}{ccc} 0&1&0 \\ 0&0&0 \\ 0&0&\lambda\end{array}\right]\,.
\]
Then
\[
\cstare(\oss_{T_\lambda})\,=\,  \left\{
       \begin{array}{lcl}
          \M_2(\CC)   &\;&      \mbox{if }\; |\lambda|\leq 1/2 \\
          \M_2(\CC)\oplus\CC &\;& \mbox{if } \; |\lambda|>1/2
      \end{array}
      \right\}
      \,.
\]
\end{example}
\begin{proof}
The C$^*$-algebra generated by $\oss_{T_\lambda}$ is $\M_2(\CC)\oplus\CC$. Let $\pi:\cstar(\oss_{T_\lambda})\rightarrow\CC$ be the
map that sends each $X\in \cstar(\oss_{T_\lambda})$ to its (3,3)-entry. Thus, $\pi$ is an irreducible representation of $\cstar(\oss_{T_\lambda})$ on the
$1$-dimensional Hilbert space $\CC$.
Another irreducible representation of $\cstar(\oss_{T_\lambda})$ is the map
$\rho:\cstar(\oss_{T_\lambda})\rightarrow\M_2(\CC)$
given by $\rho(X)=V^*XV$, where $V=\left[\begin{array}{cc} 1&0 \\ 0& 1 \\ 0&0\end{array}\right]$. Up to unitary equivalence, $\pi$ and $\rho$ are the only irreducible representations
of $\cstar(\oss_{T_\lambda})$, and so at least one of these two must be a boundary representation. In fact, regardless of the choice of $\lambda$, $\rho$ is always a boundary representation,
for it were not, then $\pi$ would necessarily be the only boundary representation for $\oss_{T_\lambda}$,
which implies that the \v Silov ideal would be given by $\mathfrak S_{\oss_{T_\lambda}}=\ker\pi=\M_2(\CC)\oplus\{0\}$;
but if this were true, then the quotient $\cstar(\oss_{T_\lambda})/\mathfrak S_{\oss_{T_\lambda}}$ would be the $1$-dimensional algebra $\CC$, which would not
contain a copy of the $3$-dimensional operator system $\oss_{T_\lambda}$. Hence, $\rho$ is a boundary representation and
the only question to resolve is: for which $\lambda$ is $\pi$ a boundary representation?
To answer this, it is enough to
use Theorem \ref{theorem: boundary reps and numerical range} and to note that the numerical range of
$\begin{bmatrix}0&1\\0&0\end{bmatrix}$ is the closed disc of radius $1/2$ centred at the origin.
\end{proof}

\section{Normal Operators and Operators with Normal $W$-Dilations}

If $T$ is a normal operator, then $\cstar(T)$ is abelian; hence, so is $\cstare(\oss_T)$, as it is
the image through an epimorphism of $\cstar(T)$. We will analyse more carefully which abelian
$\cstar$-algebras arise in such
cases, and we will show that certain non-normal $T$ have abelian $\cstar$-envelopes (even though
in these cases $\cstar(T)$ is non-abelian).

It is well known that positive maps need not be completely positive, but there is a useful ``automatic complete positivity'' result that
we will make use of.

\begin{proposition}\label{auto cp} {\rm (\cite[Theorem 3.9]{Paulsen-book})}
If $\phi:\oss\rightarrow\ost$ is a positive linear map of operator systems, and if $\ost$ is an operator subsystem $\ost\subset\csta$
of an abelian $\cstar$-algebra $\csta$, then $\phi$ is completely positive.
\end{proposition}

A \emph{function system} on a compact Hausdorff space $\Omega$ is a subset $\osf\subseteq C(\Omega)$ such that:
(i) $\osf$ is a vector space over $\CC$, closed in the topology of $C(\Omega)$;
(ii)  $f^*\in \osf$, for all $f\in \osf$;
 (iii) $1\in \osf$ (the constant function $x\mapsto 1$); and
 (iv) $\osf$ separates the points of $K$.
By the Stone--Weierstrass Theorem, the C$^*$-subalgebra of $C(\Omega)$
generated by $\osf$ is precisely $C(\Omega)$ itself.

A boundary for $\osf$ is a closed subset $\Omega_0\subseteq \Omega$ such that
for every $f\in \osf$ there is a $t_0\in \Omega_0$ such that $\|f\|=|f(t_0)|$. By a theorem of \v Silov,
there is a smallest compact subset $\shilov\osf$ of $\Omega$ that is contained in every boundary of $\osf$
and is itself a boundary of $\osf$. The set $\shilov \osf$ is known classically as the
\v Silov boundary of $\osf$.
In the language of $\cstar$-envelopes, \v Silov's theorem takes the following form:

\begin{theorem}\label{shilov} {\rm (\v Silov)} If $\osf$ is a function system on $\Omega$, then $\cstare(\osf)=C(\shilov\osf)$.
\end{theorem}

\bigskip

While for general normal operators there is a great variety of possible operator systems
and $\cstar$-envelopes, the case of selfadjoint operators is totally rigid:

\begin{proposition}\label{proposition: selfadjoint case}
If $T=T^*$, then $\cstare(\oss_T)=\CC\oplus\CC$.
\end{proposition}
\begin{proof}
One can deduce the conclusion from Proposition \ref{extremal spectra} and the fact that the numerical range of $T$ is a line segment (and thus has exactly two extreme points),
but we feel the following direct proof is more instructive.

As $\cstar(T)\simeq C(\sigma(T))$ as $\cstar$-algebras, this isomorphism restricts to a complete
isometry on $\oss_T$. So $\cstare(\oss_T)=\cstare(\oss_z)$, where $z$ is the function
$z:t\mapsto t$ in $C(\sigma(T))$.

So we want to identify the boundary representations of $\oss_z$ in $C(\sigma(T))$.
Since $C(\sigma(T))$ is an abelian
$\cstar$-algebra, each of its irreducible representation is one-dimensional, i.e. a character, and it is
given by point evaluation.

As $\sigma(T)$ is a compact subset of $\mathbb R$, it has a minimum and a maximum, say $t_0$ and $t_1$,
and every point in $\sigma(T)$ is a convex combination of $t_0$ and $t_1$.
Given any $t\in\sigma(t)$ with $t_0<t<t_1$, there exists $\alpha\in(0,1)$ with $t=\alpha t_0+(1-\alpha)t_1$. The
irreducible representation associated with $t$ is the map $\pi_t:f\mapsto f(t)$ in $C(\sigma(T))$. Now
consider the state $\psi:f\mapsto\alpha f(t_0)+(1-\alpha)f(t_1)$ on $C(\sigma(T))$. By considering
some $f\in C(\sigma(T))$ with $f(t_0)=1$, $f(t)=0$, we see that $\pi_t\ne\psi$.
But $\pi_t$ and $\psi$ agree on $\oss_z$; indeed,: if $f=\beta+\gamma z$,
\[
\psi(f)=\alpha(\beta+\gamma t_0)+(1-\alpha)(\beta+\gamma t_1)=\beta+\gamma(\alpha t_0+(1-\alpha)t_1)
=\beta+\gamma t=\pi_t(f).
\]
So $\pi_t|_{\oss_z}$ admits an extension other than $\pi_t$ (provided that $t\ne t_0,t_1$), which
shows that $\pi_t$ is not a boundary representation for $\oss_z$.

The only remaining candidates for boundary representations are $\pi_{t_0}$ and $\pi_{t_1}$. Both must be
boundary representations because the $\cstar$-envelope necessarily contains a copy of $\oss_T$ and so it has dimension
at least 2.
By Lemma \ref{lemma: the quotient as a product}, we conclude that $\cstare(\oss_z)=\CC\oplus\CC$.
\end{proof}

\begin{corollary}\label{corollary: all two-dimensional operator systems are isomorphic}
All two-dimensional operator systems are isomorphic.
\end{corollary}
\begin{proof}
It is easy to see that a two-dimensional operator system has a selfadjoint generator $T$. By Proposition
\ref{proposition: selfadjoint case}, $\cstare(\oss_T)=\CC^2$. This implies that there exists a unital complete
isometry $\psi:\oss_T\to\CC^2$. The image of $\psi$ is two-dimensional, so $\psi$ is onto,
and then $\oss_T\simeq\CC^2$
as operator systems.
\end{proof}

\bigskip

\begin{definition} Assume that $N\in\B(\H)$ is a normal operator. The \emph{function system associated with $N$}
is the operator subsystem $\osf_N\subset C(\sigma(N))$ defined by
\[
\osf_N\,=\,\mbox{\rm Span}\,\{1, \Gamma(N), \overline{\Gamma(N)}\}\,,
\]
where $\Gamma:\cstar(N)\rightarrow C(\sigma(N))$ is the Gelfand transform.
\end{definition}

\begin{proposition}\label{normal}If $N\in\B(\H)$ is normal, then $\cstare(\oss_N)=C(\shilov\osf_N)$.
\end{proposition}

\begin{proof} Note that $\cstar(\oss_N)=\cstar(N)$. The Gelfand transform $\Gamma:\cstar(N)\rightarrow C(\sigma(N))$ is an
isomorphism of $\cstar$-algebras and so
the restriction of $\Gamma$ to $\oss_N$ is a unital completely isometric linear map of $\oss_N$ onto
the operator subsystem $\osf_N\subset C(\sigma(N))$. Thus, $\Gamma|_{ \oss_N}$ is a complete order isomorphism
and, hence,
$\cstare(\oss_N)=\cstare(\osf_N)=C(\shilov\osf_N)$.
\end{proof}

\begin{corollary}\label{unitary} If $U$ is a unitary operator, then $\cstare(\oss_U)=C(\sigma(U))$.
\end{corollary}

\begin{proof} By definition, the \v Silov boundary of the function system $\osf_U$ is a compact subset of $\sigma(U)$. Therefore,
Proposition \ref{normal} shows that we need only prove the inclusion $\sigma(U)\subset\shilov\osf_U$. To this end, select $\lambda\in\sigma(U)$
and consider the function $f_\lambda\in\osf_U$ defined by
\[
f_\lambda(\mu)\,=\,\mu+\lambda\,,\;\mu\in\sigma(U)\,.
\]
For any $z\in\TT$, $|f_\lambda(z)|$ is the Euclidean distance between $z$ and $-\lambda$, and so the maximum modulus of
$f_\lambda$ on $\TT$ is attained at $\lambda$ and $|f_\lambda(\lambda)|>|f_\lambda(\mu)|$ for every $\mu\in\sigma(U)\setminus\{\lambda\}$.
Hence, $\lambda\in  \shilov\osf_U$.
\end{proof}

\begin{corollary}\label{circle} If $U$ is a unitary operator with $\sigma(U)=\TT$, then $\cstare(\oss_U)=C(\TT)$.
\end{corollary}

\bigskip

There are many operators that behave like normals when one is considering only their numerical range and spectrum. The following definition is meant
to capture such a situation.

\begin{definition} An operator $T\in\B(\H)$ has a \emph{normal $W$-dilation} if there is a Hilbert space $\K$ containing $\H$ as a subspace
and a normal operator $N\in\B(\K)$ such that:
\begin{enumerate}
\item $N$ is a dilation of $T$ (that is, $T=P_\H N|_{\H}$, where $P_\H\in\B(\K)$ is the projection of $\K$ onto $\H$), and
\item $W(T)=W(N)$.
\end{enumerate}
\end{definition}

The class of operators with normal $W$-dilations includes all Toeplitz operators on the Hardy space $H^2(\TT)$
and all subnormal operators \cite{Halmos-book}.

\begin{proposition}\label{normal dilation} If $N$ is a normal $W$-dilation of $T$, then $\oss_N$ and $\oss_T$ are completely order isomorphic.
\end{proposition}

\begin{proof} Assume that $\K\supset\H$ and that $N\in\B(\K)$ is a normal $W$-dilation of $T\in\B(\H)$.
Define $\psi:\oss_{N}\rightarrow\oss_{T}$ by $\psi(R)=P_{\H}R|_{ \H}$, which is a ucp
map that sends $N$ to $T$. Now define a linear map $\phi:\oss_{T}\rightarrow\oss_{N}$ by
\[
\phi\left( \alpha1+\beta T+\gamma T^*\right)\,=\,\alpha1+\beta N +\gamma N^*\,,\;\mbox{ for all }\alpha,\beta,\gamma\in\C\,.
\]
As a linear transformation, $\phi=\psi^{-1}$. Thus, it remains to prove that $\phi$ is completely positive. First note that the hypothesis $W(T)=W(N)$
implies that, for $R\in\oss_{T}$,
$\phi(R)$ is positive if and only if $R$ is positive. Hence, $\phi$ is a positive linear map. The range of $\phi$ is $\oss_{N}$, which is an operator subsystem of the
$\cstar$-algebra $\cstar(N)$. Because the $\cstar$-algebra $\cstar(N)$ is abelian,
all positive linear maps into $\cstar(N)$ are completely positive (Proposition \ref{auto cp}).
In particular, $\phi=\psi^{-1}$ must be completely positive, which is to say that $\psi$ is a complete order isomorphism.
\end{proof}

\begin{corollary} If $T\in\B(\H)$ is a contraction such that $\TT\subset\sigma(T)$, then $\cstare(\oss_T)=C(\TT)$.
\end{corollary}

\begin{proof} Every contraction has a unitary dilation \cite{Halmos-book}; explicitly, one such unitary dilation $U$ is given by
\[
U=\begin{bmatrix}T&(1-TT^*)^{1/2}\\ -(1-T^*T)^{1/2}&T^*\end{bmatrix}.
\]
The
condition $\TT\subset\sigma(T)$ implies, therefore, that $W(T)$ and $W(U)$ coincide with the closed unit disc and that $\sigma(U)=\TT$. Hence,
Proposition \ref{normal dilation} asserts that $\oss_U$ and $\oss_T$ are completely order isomorphic, and so $\cstare(\oss_T)=\cstare(\oss_U)$.
Corollary \ref{unitary} yields $\cstare(\oss_U)=C(\TT)$.
\end{proof}

Recall that an isometry $V$ is \emph{proper} if $V$ is not unitary.

\begin{corollary}\label{isometry} If $V$ is a proper isometry, then $\cstare(\oss_V)=C(\TT)$.
\end{corollary}

\begin{proof}  By the Wold Decomposition, the spectrum of a proper isometry $V$ necessarily contains $\TT$.
\end{proof}

As was mentioned above, for any operator $T$ one has an epimorphism $\pi:\cstar(T)\to\cstare(\oss_T)$. Whenever
this $\pi$ is not an isomorphism the \v Silov ideal, being the kernel of $\pi$, is nontrivial; in particular, $\cstar(T)$ cannot be simple.
Using this straightforward idea, we deduce the following fact from the results of this section:

\begin{corollary}
Let $T$ be an operator that is not a scalar multiple of the identity, and such that any of the following holds:
\begin{enumerate}
\item $T$ has a normal $W$-dilation;
\item $T$ is a Toeplitz operator on $H^2(\TT)$;
\item $T$ is subnormal;
\item $T$ is a contraction with $\TT\subset\sigma(T)$;
\item $T$ is a  proper isometry.
\end{enumerate}
Then $\cstar(T)$ is not simple.
\end{corollary}

\section{Finite-Dimensional Boundary Representations}

Finite-dimensional irreducible representations of $\cstar(\oss_T)$ play a role similar to that of an eigenvalue for an operator. We show in this section that such a
representation $\rho$ is a boundary representation for $\oss_T$ only if $\rho(T)$ is an extremal element in a certain convex set.

\begin{definition} Let $V$ be a complex vector space and assume that $\mathfrak K_k\subset\M_k(V)$ is a nonempty set, for every $k\in\mathbb N$. Let
$\mathfrak K=(\mathfrak K_k)_{k\in\mathbb N}$.
\begin{enumerate}
\item The sequence $\mathfrak K$ is \emph{matrix convex in $V$}
if, for every $k$,
$\displaystyle\sum_{j=1}^mA_j^*X_jA_j\in\mathfrak K_k$, whenever
$m\in\mathbb N$, $X_j\in\mathfrak K_{n_j}$, $A_j\in M_{n_j,
k}(\CC)$, and $\displaystyle\sum_{j=1}^mA_j^*A_j=1\in\M_k(\CC)$.
\item An element $X\in\mathfrak K_k$ is a \emph{matrix extreme
point} of a matrix convex set $\mathfrak K$ in $V$ if the equation
$X=\displaystyle\sum_{j=1}^mA_j^*X_jA_j$, where $X_j\in\mathfrak
K_{n_j}$, $A_j\in M_{n_j,k}(\CC)$ of rank $n_j$, and
$\displaystyle\sum_{j=1}^mA_j^*A_j=1\in\M_k(\CC)$, holds only if
each $n_j=k$ and there are unitaries $U_1,\dots, U_m\in\M_k(\CC)$
such that $X_j=U_j^*XU_j$ for all $j=1,\dots,m$.
\end{enumerate}
\end{definition}

We shall be interested in the matricial range of an operator,
which was introduced by Arveson in \cite{arveson1972} and which received subsequent study in, for example,
 \cite{arveson1972,bunce--salinas1976,smith--ward1980}.

\begin{definition} The \emph{matricial range} of an operator $T\in\B(\H)$ is the
sequence $\mathbb W(T)=\left( W_k(T)\right)_{k\in\mathbb N}$ of
subsets $W_k(T)\subset\M_k(\CC)$ defined by
\[
W_k(T)\,=\,\{\phi(T)\,:\,\phi:\oss_T\rightarrow\M_k(\CC) \mbox{ is a ucp map } \}\,.
\]
\end{definition}

It is well known that each $W_k(T)$ is compact and that $W(T)$ is matrix convex in $V=\CC$.
The set $W_1(T)$ coincides with the numerical range of $T$.

\begin{definition} If $\oss$ and $\ost$ are operator systems and $\phi,\psi:\oss\rightarrow\ost$ are completely positive linear maps such that $\phi-\psi$ is completely positive,
then $\psi$ is said to be \emph{subordinate} to $\phi$, which is denoted by $\psi\leq_{\rm cp}\phi$.
If, for given $\phi$, the only completely positive maps $\psi$ that are subordinate to $\phi$
are those $\psi$ of the form $\psi=t\,\phi$ for some $t\in[0,1]\subset\mathbb R$, then $\phi$ is said to be \emph{pure}.
\end{definition}

A completely positive linear map $\phi:\csta\rightarrow\B(\K)$, where is $\csta$ is a unital C$^*$-algebra,
is pure if and only if the representation $\pi$ that arises in the minimal Stinespring decomposition of $\phi$ is irreducible
\cite[Corollary 1.4.3]{arveson1969}. In contrast, very little can be said in general about pure maps of operator systems that are not C$^*$-algebras,
and it is in general very difficult to identify which completely positive linear maps of an operator system
are pure. However, for operator systems of the form $\oss_T$, a ucp map $\phi:\oss_T\rightarrow\M_k(\CC)$
is pure if and only if $\phi(T)\in W_k(T)$ is a matrix extreme point of
$\mathbb W(T)$ \cite[Theorem 5.1]{farenick2004}.

\begin{theorem}\label{matrix convex pure} Suppose that $T\in\B(\H)$ and that
$\rho:\cstar(T)\rightarrow\M_k(\CC)$ is an irreducible representation.
\begin{enumerate}
\item\label{mr thm 1} If $\rho$ is a boundary representation for $\oss_T$, then $\rho(T)$
is a  matrix extreme point of $\mathbb W (T)$ and $\rho|_{ \oss_{T}}$ is a pure ucp map $\oss_T\to \M_k(\CC)$.
\item\label{mr thm 2} If $\cstar(T)$ is $k$-subhomogeneous and if  $\rho|_{ \oss_{T}}$
is a pure ucp map of $\oss_T\to\M_k(\CC)$, then $\rho$ is a boundary representation for $\oss_T$.
\end{enumerate}
\end{theorem}

\begin{proof}  Assume that $\rho$ is a boundary representation for $\oss_T$.
Let $\Lambda=\rho(T)$ and suppose that
$\Lambda=\displaystyle\sum_{j=1}^mA_j^*\Omega_jA_j$ for $\Omega_j\in W_{n_j}(T)$
and $n_j\times k$ matrices  $A_j$ of rank $n_j$ satisfying $\displaystyle\sum_{j=1}^mA_j^*A_j=1$.
As $\Omega_j\in W_{n_j}(T)$, there are
ucp maps $\phi_j:\cstar(T)\rightarrow\M_{n_j}(\CC)$ such that $\phi_j(T)=\Omega_j$,
and so the matricial state
$\phi=\sum_j A_j^*\phi_j A_j$, whereby $X\mapsto \sum_j A_j^*\phi_j(X)A_j$,
is a ucp extension of $\rho{}|_{ \oss_T}$. By hypothesis, $\phi$ must equal $\rho$; hence, for each $j$,
\[
A_j^*\phi_j A_j \,\leq_{\rm cp}\,\rho\,.
\]
Since $\rho$ is an irreducible representation, it is pure as a completely
positive linear map of $\cstar(T)$ into $\M_k(\CC)$ \cite[Corollary 1.4.3]{arveson1969}.
Thus, there are $t_j\in[0,1]$
such that
\[
A_j^*\phi_j A_j \,=\,t_j\rho\,.
\]
Let $U_j=t_j^{-1/2}A_j$. Then evaluation at $1\in\cstar(T)$ gives us $U_j^*U_j=1\in M_k(\CC)$. So
$U_j$ is isometric and has rank $k$; we knew that $A_j$ (and so $U_j$) has rank $n_j$, and
we conclude that $n_j=k$. Then $U_j\in\M_{k}(\CC)$ is a unitary. But
$U_j^*\phi_j U_j=\rho$ implies that $\Omega_j=U_j\Lambda  U_j^*$ for each $j$, which shows that
$\Lambda$ is a matrix extreme point of $\mathbb W(T)$.
Therefore, by \cite[Theorem 5.1]{farenick2004}, $\rho|^{\phantom{T}}_{\oss_{T}}$
is a pure ucp map $\oss_T\to\M_k(\CC)$.

Conversely, suppose that $\cstar(T)$ is $k$-subhomogeneous and that
$\rho|^{\phantom{T}}_{\oss_{T}}$ is a pure ucp map of $\oss_T\to\M_k(\CC)$. Let
$C_\rho$ be the BW-compact, convex set of of all ucp maps $\psi:\cstar(T)\rightarrow\M_{k}(\CC)$
that extend $\rho|^{\phantom{T}}_{\oss_{T}}$. By the proof of \cite[Theorem B]{farenick2000},
every extreme point $\phi$ of $C_\rho$ is a pure matrix state of $\cstar(T)$, and so we need only show that the only pure extension $\phi$ of $\rho|_{ \oss_T}$ to $\cstar(T)$
is $\phi=\rho$. To this end, let $\phi=v^*\pi v$ be a minimal Stinespring decomposition of $\phi$, where $\pi:\cstar(T)\rightarrow\B(\H_\pi)$ is a representation and
$v:\CC^{k}\rightarrow\H_\pi$ is an isometry. Because $\phi$ is pure, $\pi$ is necessarily irreducible
\cite[Corollary 1.4.3]{arveson1969}. Hence, $\dim\H_\pi\leq k$, as $\cstar(T)$ is $k$-subhomogeneous. But because $v$ is an isometry, necessarily $\dim\H_\pi=k$.
Thus, $v$ is a unitary;
it follows that $\phi=v^*\pi v$ is multiplicative; as it agrees with $\rho$ in the generating
set $\oss_T$, we get that $\phi=\rho$.
\end{proof}

An operator $T\in\B(\H)$ is \emph{$k$-normal} if any elements
$X_1,\dots, X_{2k}$ in the von Neumann algebra $\mathcal N_T$ generated by $T$, satisfies
\[
\sum_{\tau\in \mathbb S_{2k}}\,{\epsilon(\tau)} X_{\tau(1)}\cdots X_{\tau(2k)}\,=\,0,
\]
where $\mathbb S_{2k}$ denotes the group of permutations on $\{1,\dots,2k\}$
and $\epsilon(\tau)$ denotes the parity (even or odd) of a permutation $\tau$.
Because the C$^*$-algebra
generated by a $k$-normal operator is $k$-subhomogeneous \cite{bunce--deddens1972}, we obtain the following result:

\begin{corollary}\label{k-normal} If $T$ is a $k$-normal operator, then the
following statements are equivalent for a representation $\rho:\cstar(T)\rightarrow\M_k(\CC)$:
\begin{enumerate}
\item\label{k-normal1} $\rho$ is a boundary representation for $\oss_T$;
\item\label{k-normal2} $\rho(T)$ is a matrix extreme point of $\mathbb W(T)$;
\item\label{k-normal3} $\rho|_{\oss_T}$ is a pure ucp map.
\end{enumerate}
\end{corollary}

\section{Irreducible Periodic Weighted  Shift Operators}

In this section we present the main result (Theorem \ref{theorem: periodic shift})
of the paper. The operators we consider are irreducible periodic weighted unilateral shifts on $\ell^2(\mathbb N)$; however, it is instructive to
consider first the case of unilateral weighted shifts on finite-dimensional Hilbert spaces.

\begin{definition} If $\CC^*:=\CC\setminus\{0\}$ and
$\xi=\displaystyle\sum_{i=1}^d \xi_ie_i\in(\CC^*)^d$, then the \emph{irreducible weighted unilateral shift} with weights $\xi_1,\dots,\xi_d$ is the operator $W(\xi)$ on $\CC^{d+1}$ given by the matrix
\[
W(\xi)\,=\, \left[\begin{array}{ccccc}
0 &&&& 0 \\
 \xi_1 & 0 &&& \\
 & \xi_2 & \ddots && \\
 &&\ddots &0 & \\
 &&& \xi_{d} & 0 \end{array} \right] \,.
 \]
\end{definition}

\begin{proposition} The C$^*$-envelope of an irreducible weighted unilateral shift acting on $\CC^{d+1}$ is $\M_{d+1}(\CC)$. Furthermore, if
$\xi,\eta\in(\CC^*)^d$, then the operator systems $\oss_{W(\xi)}$ and $\oss_{W(\eta)}$ are unitally completely order isomorphic if and only if $|\xi|=|\eta|$,
where, for $\nu\in\CC^d$, $|\nu|\in\mathbb R_+^d$ denotes the vector of moduli of the coordinates of $\nu$.
\end{proposition}

\begin{proof} If $\xi\in(\CC^*)^d$, then the operator system $\oss_{W(\xi)}$ is irreducible and,
hence, $\cstar(\oss_{W(\xi)})=\M_{d+1}(\CC)$, which is simple.
Therefore, the \v Silov boundary ideal for $\oss_{W(\xi)}$ is necessarily
trivial and so $\cstare(\oss_{W(\xi)})=\cstar(\oss_{W(\xi)})=\M_{d+1}(\CC)$.

Assume now that there is a unital complete order isomorphism
$\phi:\oss_{W(\xi)}\rightarrow\oss_{W(\eta)}$. As both $W(\xi)$ and $W(\eta)$ are irreducible,
$\phi$ is necessarily implemented by an automorphism of $\M_{d+1}(\CC)$
\cite[Theorem 0.3]{arveson1972}; that is, there is a unitary $U$ such that $W(\xi)=U^*W(\eta)U$.
But $W(\xi)$
and $W(\eta)$ are unitarily similar if and only if $|\xi|=|\eta|$ (by direct computation or by applying \cite[Theorem 3.2]{farenick--gerasimova--shvai2011}).
\end{proof}

Returning to the case of irreducible $p$-periodic weighted unilateral shifts on $\ell^2(\mathbb N)$,
the image of any such
operator in the Calkin algebra generates a $p$-homogeneous C$^*$-algebra, and in this case Theorem \ref{matrix convex pure} (or Corollary \ref{k-normal}) could be
invoked. However, Theorem \ref{matrix convex pure} is an abstract characterisation
which yields limited information in specific cases. Therefore, this section aims
to give full information about the noncommutative Choquet boundary and the C$^*$-envelope of $\oss_W$ for
irreducible periodic weighted unilateral shifts $W$.

 \begin{definition} A \emph{weighted unilateral shift operator} is an operator $W$ on $\ell^2(\mathbb N)$
defined on the standard orthonormal basis $\{e_n\,:\,n\in\mathbb N\}$ of $\ell^2(\mathbb N)$ by
\[
We_n=w_ne_{n+1}\,,\;n\in\mathbb N,
\]
where the \emph{weight sequence} $\{w_n\}_{n\in\mathbb N}$ for $W$ consists of
nonnegative real numbers with $\sup_n w_n<\infty$. If there is a $p\in\mathbb N$
such that $w_{n+p}=w_n$ for every $n\in\mathbb N$,
then $W$ is called a \emph{periodic}
unilateral weighted shift of \emph{period} $p$. If at least one of $w_1,\ldots,w_p$ is not repeated in the
list, we say that $W$ is \emph{distinct}.
\end{definition}

Proposition \ref{isometry}
demonstrates that the C$^*$-envelope of the operator system $\oss_W$ generated by a
periodic unilateral weighted shift operator $W$ of period $p=1$
is the abelian $C^*$-algebra $C(\TT)$.
To determine the C$^*$-envelope of an irreducible periodic unilateral weighted shift operator of period $p>1$,
a notion related to matrix convexity comes into play.

\begin{definition} Assume that $\mathfrak C\subset\M_k(\CC)$ is a nonempty set.
\begin{enumerate}
\item $\mathfrak C$ is \emph{C$^*$-convex} if $\displaystyle\sum_{j=1}^mA_j^*X_jA_j\in\mathfrak C$
for every $m\in\mathbb N$, $X_1,\dots, X_m\in\mathfrak C$, and $A_1,\dots, A_m\in\M_k(\CC)$ satisfying $\displaystyle\sum_{j=1}^mA_j^*A_j=1$.
\item An element $X\in\mathfrak C$ is a
\emph{C$^*$-extreme point} of a C$^*$-convex set $\mathfrak C$ if the equation $X=\displaystyle\sum_{j=1}^mA_j^*X_jA_j$, for
$X_1,\dots, X_m\in\mathfrak C$ and invertible $A_1,\dots, A_m\in\M_k(\CC)$ with
$\displaystyle\sum_{j=1}^mA_j^*A_j=1$, implies that there exist
unitaries $U_1,\dots, U_m\in\M_k(\CC)$ such that $X_j=U_j^*XU_j$ for all $j=1,\dots,m$.
\end{enumerate}
\end{definition}

\bigskip

\begin{theorem}\label{theorem: periodic shift} Assume that f $W\in\B(\ell^2(\mathbb N))$ is an
irreducible periodic distinct unilateral weighted shift with smallest period $p$. Then
$\cstare(\oss_W)=C(\TT)\otimes \M_p(\CC)$ and $\mathfrak S_{\oss_W}=\K\left(\ell^2(\mathbb N)\right)$.
\end{theorem}

\begin{proof} We will assume that $w_p\not\in\{w_1,\ldots,w_{p-1}\}$; one such weight exists by $W$
being distinct; we will assume that it is $w_p$ because it simplifies the writing a little, but
the same idea can be used with any other weight. By periodicity and the fact that
$\ell^2(\mathbb N)\cong\displaystyle\bigoplus_1^p\ell^2(\mathbb N)$, we may
express $W$ as $p\times p$ matrix of operators acting on $\ell^2(\mathbb N)$  \cite[first paragraph in the proof
of Theorem 2.2]{bunce--deddens1973}:
\[
W\,=\,\left[\begin{array}{ccccc}
0 &&&& w_pS \\
 w_11 & 0 &&& \\
 & w_21 & \ddots && \\
 &&\ddots &0 & \\
 &&& w_{p-1}1 & 0 \end{array} \right] \,,
 \]
where unspecified entries of the matrix above are zero and $S\in\B\left(\ell^2(\mathbb N)\right)$ denotes
the unilateral shift operator.
The operator system $\oss_W$ is an operator subsystem of $\oss_S\otimes\M_p(\CC)$.

We aim to show first that $\cstar(W)=\cstar(\oss_S\otimes\M_p(\CC))$.
Of course we already have the inclusion $\cstar(W)\subset\cstar(\oss_S\otimes\M_p(\CC))$,
and so we consider the converse by a method suggested by the
proof of \cite[Proposition V.3.1]{Davidson-book}.
Note that $\cstar(\oss_S\otimes\M_p(\CC))=\cstar(\oss_S)\otimes\M_p(\CC)$.
Let $\{E_{ij}\}_{i,j=1}^p\subset\M_p(\CC)$ be the standard matrix units for $\M_p(\CC)$, and let
$F_{ij}=1\otimes E_{ij}\in \cstar(S)\otimes\M_p(\CC)$.
Because $W$ is irreducible, $w_k>0$ for all $k$.
Note that $|W|=(W^*W)^{1/2}\in\cstar(W)$
is the diagonal operator matrix $|W|=\sum_{k=1}^p w_kF_{kk}$.
Now let
$f\in\CC\,[t]$ be any polynomial for which $f(w_1)=\cdots=f(w_{p-1})=0$ and $f(w_p)=1$ (here is where
we use that $W$ is distinct); then
$F_{pp}=f(|W|)\in \cstar(W)$.

Now for any $i,j\in\{1,\dots,p\}$,
\[
(W^*)^{p-i}F_{pp}W^{p-j}\,=\,\alpha_{ij}F_{ij}\,,
\]
where $\alpha_{ij}>0$ is a product of weights $w_\ell$. Thus, $\cstar(W)$ contains each of the matrix units $F_{ij}$. Moreover,
$S\otimes E_{11} = \frac{1}{w_p} F_{1p}WF_{p1}\in\cstar(W)$. By multiplying $S\otimes E_{11} $ on the left and right with appropriate
matrix units $F_{ij}$ we obtain $S\otimes E_{ij}\in\cstar(W)$ for every $i$ and $j$.
Hence, $\oss_S\otimes\M_p(\CC)\subset\cstar(W)$ and so
$\cstar(\oss_S\otimes\M_p(\CC))=\cstar(W)$.

Because $S$ is a proper isometry, Proposition \ref{isometry} states that $\cstare(\oss_S)=C(\TT)$.
Hence, there is an epimorphism
$\pi:\cstar(\oss_S)\rightarrow C(\TT)$ such that $\pi|_{ \oss_S}$ is a completely isometric linear map
that maps $S$ to the function $z\in C(\TT)$ given by $z(e^{i\theta})=e^{i\theta}$. Due to the fact that
$C(\TT)$ is abelian, it is easy to see that $\pi=0$ when restricted to the compact operators.
Let $\rho=\pi\otimes\mbox{id}_{\M_p}$, which is an epimorphism of $\cstar(\oss_S)\otimes\M_p(\CC)$
onto $C(\TT)\otimes\M_p(\CC)$ such that
$\rho|_{ \oss_S\otimes\M_p(\CC)}$ is a unital completely isometric map. Therefore,
$\iota:=\rho|^{\phantom{\oss_W}}_{\oss_W}$ is a completely isometric
embedding of $\oss_W$ into $C(\TT)\otimes\M_p(\CC)$:
\[
\oss_W\longrightarrow  \cstar(\oss_W)  \longrightarrow C(\TT)\otimes\M_p(\CC)\,.
\]
Under this embedding $\iota$, $W$ is mapped to the matrix
\[
\iota(W)\,=\, \left[\begin{array}{ccccc}
0 &&&& w_pz \\
 w_1 & 0 &&& \\
 & w_2 & \ddots && \\
 &&\ddots &0 & \\
 &&& w_{p-1} & 0 \end{array} \right] \,.
 \]
Because $\rho$ is onto, the $\cstar$-algebra $C(\TT)\otimes \M_p(\CC)$ is
generated by $\oss_{\iota(W)}$, the completely isomorphic copy of $\oss_W$.

Hence, we need no longer work with $W$ and $\cstar(W)$, but may instead study $\oss_{\iota(W)}$ and $\cstar(\iota(W))$. In this regard,
we show that the \v Silov boundary ideal of $\oss_{\iota(W)}$ is $\{0\}$, which implies that
\[
\cstare(\oss_W)=\cstare(\oss_{\iota(W)})=\cstar(\oss_{\iota(W)})=C(\TT)\otimes\M_p(\CC)\,.
\]
This is achieved by showing that
every irreducible representation of $C(\TT)\otimes\M_p(\CC)$ is a boundary representation for $\oss_{\iota(W)}$.

To this end, observe first that the irreducible representations of $C(\TT)\otimes\M_p(\CC)$ are determined by points $\lambda\in\TT$ and are
of the form
\[
\xymatrix{
**[r]\pi_\lambda: C(\TT)\otimes\M_p(\CC)\ \ \ \ \ar[rrr] & & & \ \ \ \ \ \ \M_p(\CC)}
\]
\[
\xymatrix{
**[r]\ \ \ \ \ \ \ \ \ \ \ \ \ \ [f_{kj}]_{k,j=1}^p \ \ \ \ \ \ \ \ \   \ar@{|->}[rrr]& & & \ \ \ \ \ [f_{kj}(\lambda)]_{k,j=1}^p
}
\]
where $f_{kj}\in C(\TT)$. For each $\lambda\in\TT$ let $\Omega_\lambda\in\M_p(\CC)$ denote the (irreducible) matrix $\Omega_\lambda=\pi_\lambda(\iota(W))$.
By \cite[Theorems 3.9, 3.10]{bunce--salinas1976}, the C$^*$-convex hull of the
set $\{\Omega_\lambda\,:\,\lambda\in\TT\}$ is precisely
the set $\mathfrak W_p$ of all matrices of the form $\Phi(\iota(W))$, where
$\Phi:C(\TT)\otimes\M_p(\CC)\rightarrow\M_p(\CC)$ is an arbitrary ucp map, i.e.
\[
\mathfrak W_p=\{\Phi(\iota(W)):\ \Phi:C(\TT)\otimes\M_p(\CC)\rightarrow\M_p(\CC) \text{ ucp }\}.
\] Because
every $\Omega_\lambda$ is irreducible, every structural element of $\mathfrak W_p$ is unitarily equivalent to some $\Omega_\lambda$, by Morenz's
Krein--Milman Theorem \cite[Theorem 4.5]{morenz1994}. Hence, for at least one $\lambda_0\in\TT$ the matrix
$\Omega_{\lambda_0}$ is a
C$^*$-extreme point of $\mathfrak W_p$.
We now show that for this particular $\lambda_0$ the irreducible representation $\pi_{\lambda_0}$
is a boundary representation for $\oss_{\iota(W)}$.

The BW-compact set $C_{\lambda_0}$ of all ucp maps $\psi:C(\TT)\otimes\M_p(\CC)\rightarrow\M_p(\CC)$
that extend $\pi_{\lambda_0}|_{\oss_{\iota(W)}}$ is convex; thus,
it is sufficient to show that if $\phi$ is an extreme point of $C_{\lambda_0}$, then $\phi=\pi_{\lambda_0}$.
Because C$^*$-extreme points of matrix sets are also extreme points, $\Omega_{\lambda_0}$
is an extreme point of $\mathfrak W_p$. Hence, by a standard convexity argument,
the extreme point $\phi$ of $C_{\lambda_0}$ is also an extreme point of the
set of all ucp maps $\vartheta:C(\TT)\otimes\M_p(\CC)\rightarrow\M_p(\CC)$. Now we write $\phi=V^*\pi V$ using
a minimal Stinespring decomposition, where $V$ is an isometry $\CC^p\to \H_\pi$
and $\pi:\C(\TT)\otimes M_p(\CC)\to B(\H_\pi)$
for some Hilbert space $\H_\pi$. By
\cite[Theorem 1.4.6]{arveson1969}, the subspace $V\CC^p\subset\H_\pi$ is
faithful for the commutant of
$\pi(C(\TT)\otimes M_p(\CC))$. Hence, $\pi(C(\TT)\otimes M_p(\CC))V\CC^p$ is dense in $\H_\pi$, and
so $\H_\pi$ is finite-dimensional.
Thus, we can write
$\pi=\displaystyle\bigoplus_{j=1}^m\pi_j$ as a decomposition into a finite
direct sum of irreducible (sub)representations $\pi_j$, where $\H_{\pi_j}\subset\H_\pi$
is a subspace. Then each $P_j=\pi_j(1)$ is a central
projection in $\pi(C(\TT)\otimes M_p(\CC))^\prime$, and
$\displaystyle\sum_{j=1}^mP_j=1$ in $\B(\H_\pi)$.

Because the spectrum of the C$^*$-algebra $C(\TT)\otimes M_p(\CC)$ is $\mathbb T$, for each $j=1,\dots,m$ there is a $\lambda_j\in\mathbb T$ such that $\pi_j=\pi_{\lambda_j}$.
Therefore, we can write, for $f\in C(\TT)\otimes M_p(\CC)$,
\[
\phi(f)=V^*\pi(f)V=V^*\left(\sum_{j=1}^m \pi_{\lambda_j}(f)P_j\right)V=\sum_{j=1}^m (P_jV)^*\pi_{\lambda_j}(f) (P_jV)
\]
Note that
\begin{equation}\label{equation: convex coefficients}
\sum_{j=1}^m(P_jV)^*(P_jV)=\sum_{j=1}^m V^*P_jV=V^*V=1.
\end{equation}
If $\xi\in\CC^p_{\phantom{p}}$ is a unit vector we define, for nonzero $P_jV\xi$,
$\hat\xi_j=\|P_jV\xi\|^{-1}P_jV\xi$; otherwise we let $\hat\xi_j=0$. Then
\begin{align*}
\langle \Omega_{\lambda_0}\xi,\xi\rangle\,&=\,\langle \phi(\iota(W))\xi,\xi\rangle\,
=\sum_{j=1}^m\langle \pi_{\lambda_j}(\iota(W))P_jV\xi,P_jV\xi\rangle \\
&=\,\sum_{j=1}^m \|P_jV\xi\|^2\langle\Omega_{\lambda_j}\hat\xi_j,\hat\xi_j\rangle\,.
\end{align*}

The equality in \eqref{equation: convex coefficients} implies that $\sum_{j=1}^m\|P_jV\xi\|^2=1$ (i.e. they are
convex coefficients), and so we obtain
\begin{equation}\label{incl}
W(\Omega_{\lambda_0})\,\subset\,{\rm Conv}\,\left(\bigcup_{j=1}^m W(\Omega_{\lambda_j})\right)\,.
\end{equation}
If $\zeta,\nu\in\TT$ are arbitrary, then
the moduli of the weights in the shift matrices $\Omega_{\zeta}$ and $\Omega_\nu$ coincide; thus,
$\Omega_{\zeta}$ and $\Omega_\nu$ have the same numerical radius \cite[Lemma 2(2)]{tsai--wu2011}. Hence, there
is a constant $r>0$ such that the numerical radius of $\Omega_\zeta$ is $r$ for every $\zeta\in\TT$.
Furthermore, for any $\zeta\in\TT$,
\[
W(\Omega_\zeta)\,\cap\,r\TT\,=\,\{\omega^k\zeta\,:\,k=1,\dots,p\}\,,
\]
where $\omega\in\CC$ is a primitive $p$-th root of unity \cite[Proposition 3]{tsai--wu2011}. Thus, there are exactly $p$ extreme points
of the numerical range of any $\Omega_\zeta$ on the circle $r\TT$. Hence, the only way in which the inclusion \eqref{incl} can hold
is if $\lambda_j=\lambda_0$ for every $j$. Consequently,
\[
\Omega_{\lambda_0}\,=\,\phi(\iota(W))\,=\,\sum_{j=1}^m (P_jV)^*\pi_{\lambda_0}(\iota(W)) P_jV\,
=\,\sum_{j=1}^m (P_jV)^*\Omega_{\lambda_0} P_jV\,.
\]
Now because $\Omega_{\lambda_0}$ is an irreducible C$^*$-extreme point of $\mathfrak W_p$,
the expression above for $\Omega_{\lambda_0}$
holds only if there are
unitaries $U_1,\dots, U_m\in\M_p(\CC)$ and convex coefficients $t_j\in(0,1)$ such that $P_jV=t_j^{1/2}U_j$ \cite[Corollary 1.8]{morenz1994}.
Thus,
\[
\Omega_{\lambda_0}\,=\,\sum_{j=1}^mt_j U_j^*\Omega_{\lambda_0} U_j\,.
\]
However, every matrix is an extreme point of the convex hull of its unitary orbit and
so $U_j^*\Omega_{\lambda_0} U_j=\Omega_{\lambda_0}$ for each $j$.
As $\Omega_{\lambda_0}$ is irreducible, each $U_j$ is the identity and so
\[
\phi\,=\,\sum_{j=1}^m (P_jV)^*\pi_{\lambda} P_jV\,=\,\sum_{j=1}^mt_j U_j^*\pi_{\lambda_0}  U_j\,
=\,\sum_{j=1}^mt_j\pi_{\lambda_0}\,=\,\pi_{\lambda_0}\,.
\]
This completes the proof that $\pi_{\lambda_0}$ is a boundary representation for at least one $\lambda_0\in\TT$.

Note that we have
\[
\mathfrak W_p=\{\tilde\Phi(W):\ \tilde\Phi:\cstar(W)\to M_p(\CC), \text{ucp, }\tilde\Phi(K)=0\ \forall
K\in\cstar(W)\cap\K(\H)\}.
\]
Indeed, for $\tilde\Phi$ as above, using that $\iota$ is a
complete isometry we can define the ucp map
$\Phi=\tilde\Phi\circ\iota^{-1}:\oss_{\iota(W)}\to M_p(\CC)$ and then extend by Arveson's Extension Theorem
to $\cstar(\iota(W))=C(\TT)\otimes M_p(\CC)$. By construction, $\Phi(\iota(W))=\tilde\Phi(W)$.
Conversely, if $\Phi:C(\TT)\otimes M_p(\CC)\to M_p(\CC)$ is ucp, we can define $\tilde\Phi=\Phi\,\circ\,\iota$.
As both $\Phi$ and $\iota$ are ucp, so is $\tilde\Phi$ and we can extend it using Arveson's Extension Theorem
to all of $\cstar(W)$.  For any $K\in\cstar(W)\cap\K(\H)$ we have $\iota(K)=0$, and so by
construction, $\tilde\Phi(K)=0$.

If $\theta\in\mathbb R$, then $e^{i\theta}W$ is unitarily
equivalent to $W$; indeed, if we note that $e^{ip\theta}S$ is an
isometry, then by the Wold decomposition there is a unitary $V$
such that $e^{ip\theta}S=VSV^*$ (one can write this unitary explicitly: it is the diagonal unitary
in $\B(\ell^2(\NN))$
with diagonal $(1,e^{ip\theta},e^{2ip\theta},\ldots)$). Let $U$ be the block-diagonal unitary
\[
U=\begin{bmatrix} V\\ & e^{i\theta}V\\ & & e^{2i\theta}V\\ &  & & \ddots \\ & & & & e^{(p-1)i\theta}V
\end{bmatrix}
\]
A straightforward computation then shows that $UW=e^{i\theta}WU$, and so $UWU^*=e^{i\theta}W$.
We conclude that $\mathfrak W_p$ is closed under multiplication by scalars of modulus $1$.

Now select an arbitrary $\lambda'\in\TT$. We aim to show that $\pi_{\lambda'}$ is a boundary representation.
To do so, by the method of proof above applied to $\pi_{\lambda_0}$
it is sufficient to show that $\Omega_{\lambda'}$ is an irreducible C$^*$-extreme point of
$\mathfrak W_p$. Because the weighted shift matrix $\Omega_{\lambda'}$ differs
from $\Omega_{\lambda_0}$ in the $(1,p)$-entry only, and because $|\lambda'|=|\lambda_0|$,
there are a unitary $U'\in\M_p(\CC)$ and a $\theta\in\mathbb R$ such that
$e^{i\theta}\Omega_{\lambda'}=(U')^*\Omega_{\lambda_0} U'$ \cite[Lemma 2(2)]{tsai--wu2011}.
As C$^*$-extreme points are closed under unitary similarity and because $\Omega_{\lambda_0}$
is C$^*$-extremal in $\mathfrak W_p$, we deduce that $e^{i\theta}\Omega_{\lambda'}$ is a C$^*$-extreme point of $\mathfrak W_p$. That is, $\Omega_{\lambda'}$
is a C$^*$-extreme point of $e^{-i\theta}\mathfrak W_p=\mathfrak W_p$.

Hence, the boundary representations for $\oss_W$ are precisely the irreducible representations of
$\cstar(W)$ of the form $\pi_\lambda\circ \pi$, for all $\lambda\in\TT$,
which is to say that $\cstare(\oss_W)=C(\TT)\otimes \M_p(\CC)$ and $\mathfrak S_{\oss_W}=\K\left(\ell^2(\mathbb N)\right)$.
 \end{proof}


\end{document}